\documentclass[a4,11pt]{article}
\usepackage{latexsym}
\usepackage{amsmath}
\usepackage{amssymb}
\usepackage{amsfonts}
\usepackage{amsthm}
\usepackage{xspace}

\usepackage{epsfig}

\usepackage[dvips]{color}
\usepackage{epsfig}
\usepackage{graphics}
\usepackage{graphicx}

\newtheorem{thm}{Theorem}

\newtheorem{cor}[thm]{Corollary}

\newtheorem{proposition}[thm]{Proposition}

\newtheorem{example}{Example}

\newtheorem{lem}[thm]{Lemma}

\DeclareMathOperator{\pw}{pwd}
\DeclareMathOperator{\dc}{dc}
\DeclareMathOperator{\re}{re}
\DeclareMathOperator{\bw}{bw}

\newcommand{\bR}{\mathbb{R}}

\begin{document}

\begin{center}
{\LARGE A note on dilation coefficient, plane-width,\\
        and resolution coefficient of graphs}

~\\

{{\bf Martin Milani\v c},\\
University of Primorska---FAMNIT, Glagolja\v ska 8, 6000 Koper, Slovenia,
martin.milanic@upr.si
\\
{\bf Toma\v z Pisanski},\\
University of Ljubljana, IMFM, Jadranska 19, 1111 Ljubljana, Slovenia,
and\\
University of Primorska---PINT, Muzejski trg 2, 6000 Koper, Slovenia, \\
Tomaz.Pisanski@fmf.uni-lj.si\\
{\bf Arjana \v Zitnik}\\
University of Ljubljana, IMFM, Jadranska 19, 1111 Ljubljana, Slovenia,
Arjana.Zitnik@fmf.uni-lj.si}
\end{center}

\medskip

\noindent{\bf Math. Subj. Class. (2000) 05C62}

\begin{abstract}
In this note we study and compare three graph invariants related to the
`compactness' of graph drawing in the plane: the {\em dilation coefficient}, defined as the smallest possible quotient between the longest and the shortest edge length;
the \emph{plane-width}, which is the smallest possible quotient between the largest
distance between any two points and the shortest length of an edge; and
the \emph{resolution coefficient}, the smallest possible quotient between
the longest edge length and the smallest distance between any two points.
These three invariants coincide for complete graphs.

We show that graphs with large dilation coefficient or plane-width have a vertex with large valence but there exist cubic graphs with arbitrarily large resolution coefficient. Surprisingly enough, the one-dimensional analogues of these three invariants allow us to revisit the three well known graph parameters: the circular chromatic number, the chromatic number, and the bandwidth. We also examine the connection between bounded resolution coefficient and minor-closed graph classes.
\end{abstract}

\section{Introduction}
\begin{sloppypar}Given a simple, undirected, finite graph $G = (V,E)$, a {\em representation} of $G$
in the  $d$-dimensional Euclidean space  $\bR^d$ is a function
$\rho$ assigning to each vertex of $G$ a point in $\bR^d$. An $\bR^2$-representation is called {\em planar}, an $\bR^3$-representation
is called {\em spatial.} A  representation of a graph  in $\bR^d$ can be viewed
as a drawing of the graph  with straight edges. The length of an edge
$uv \in E$ is the Euclidean distance between $\rho(u)$ and $\rho(v)$, i.e.,
$\|\rho(u)- \rho(v)\|_2$.
A representation  $\rho$  is {\em non-vertex-degenerate} (NVD) if the function $\rho$ is one-to-one. A representation  $\rho$  is {\em non-edge-degenerate} (NED) if
for all edges $uv\in E(G)$, it holds that $\rho(u)\neq\rho(v)$;
in other words, if the minimal edge length is positive.
Clearly, a non-vertex-degenerate representation is also non-edge-degenerate.
\end{sloppypar}

In this paper we study three graph invariants that are relevant for
representations of graphs. We restrict ourselves to planar representations.
To avoid trivialities, we only consider graphs with at least one edge.

Following~\cite{Pisanski.Zitnik.2008}, we define the \emph{dilation coefficient}
of a non-edge-degenerate representation of a graph as the ratio between the longest
and the shortest edge length in the representation. The minimum of the set of dilation coefficients of all non-edge-degenerate representations of a graph $G$ is called the \emph{dilation coefficient} of $G$, and is denoted by $\dc(G)$.

The \emph{plane-width} of a graph $G=(V,E)$, introduced
in~\cite{Kaminski.et.al.2009} and denoted by $\pw(G)$,
is the minimum diameter of the image of the graph's vertex set,
over all representations $\rho$ with the property that
$d (\rho(u), \rho(v)) \geq 1$ for each edge $uv\in E$, where $d$ is
the Euclidean distance. Equivalently, the plane-width of a graph can be
defined as the minimum, over all non-edge-degenerate representations,
of the ratio between the largest distance between two points and
the shortest length of an edge.

We define the \emph{resolution coefficient} of a non-vertex-degenerate representation
of a graph as the ratio between the longest length of an edge and
the smallest distance between the images of any two distinct vertices.
The minimum of the set of resolution coefficients of  all representations
of a graph $G$ is called the \emph{resolution coefficient} of $G$,
and is denoted by $\re(G)$.

In formulae, denoting by $V_2(G)$ the set of all 2-element subsets of $V(G)$, we have
\begin{equation}                                               \label{eq:dc}
\dc(G)=\min\left\{\frac{\max_{uv\in
E(G)}d(\rho(u),\rho(v))} {\min_{uv\in E(G)}d(\rho(u),\rho(v))}\,:\, \rho \textrm{ is an NED representation of }G\right\} \,,
\end{equation}

\begin{equation}                                               \label{eq:pw}
\pw(G)=\min\left\{\frac{\max_{uv\in
V_2(G)}d(\rho(u),\rho(v))} {\min_{uv\in E(G)}d(\rho(u),\rho(v))}\,:\, \rho \textrm{ is an NED representation of }G \right\} \,,
\end{equation}

\begin{equation}                                              \label{eq:res}
\re(G)=\min\left\{\frac{\max_{uv\in
E(G)}d(\rho(u),\rho(v))} {\min_{uv\in V_2(G)}d(\rho(u),\rho(v))}\,:\, \rho \textrm{ is an NVD representation of }G \right\} \,.
\end{equation}

Note that the above parameters are well-defined, as we only consider graphs with
at least one edge, and each of the three parameters can be expressed as the
minimum value of a continuous real-valued function over a compact subset of $\bR^{2|V(G)|}$.

Notice also that considering the smallest ratio between the largest and
the smallest distance between any two distinct points in a non-vertex-degenerate representation does not lead to a meaningful graph parameter -- indeed, this quantity depends only on the number of vertices of the graph and not on the graph itself.
The problem of determining the minimum value of this parameter
for a given number of points $n$ has previously appeared in the literature
in different contexts such as finding the minimum diameter of a~set of
$n$ points in the plane such that each pair of points is at distance at least
one \cite{Bezdek.Fodor.1999}, or packing non-overlapping unit discs in the plane
so as to minimize the maximum distance between any two disc centers
\cite{Schurmann.2002}.
For a given $n$, we denote the value of this parameter by $h(n)$:
\begin{equation*}                                             
h(n)=\min\left\{\frac{\max_{uv\in V_2(K_n)} d(\rho(u),\rho(v))}
                     {\min_{uv\in V_2(K_n)}d(\rho(u),\rho(v))}\,:\,
                     \rho \textrm{ is an NVD representation of }K_n \right\} \,.
\end{equation*}

\begin{table}[h]
\begin{center}
\begin{tabular}{|c||c|c|c|c|c|c|c|}
\hline
$n$ &2 & 3 & 4 & 5 & 6 & 7 & 8 \\
\hline
$h(n)$ & 1 & 1 & $\sqrt{2}$ & $\frac{1 + \sqrt{5}}{2}$ & $2\sin 72^\circ$ & 2 & $(2 \sin(\pi/14))^{-1}$\\
\hline
approx. & 1 & 1 & 1.414 & 1.618 & 1.902 & 2 & 2.246 \\
\hline
\end{tabular}
\caption{Known values of $h(n)=\dc(K_n)=\pw(K_n)=\re(K_n)$.}
\label{table:valuesKn}
\end{center}
\end{table}

The exact values for $h(n)$ are known only up to $n=8$
\cite{Bateman.Erdos.1951, Bezdek.Fodor.1999}, see Table \ref{table:valuesKn}.
It follows directly from the definitions that if $G$ is a complete graph, the
three parameters  $\dc$, $\pw$ and $\re$ coincide.

\begin{proposition}\label{prop:0}
For every complete graph $K_n$,
$$\dc(K_n)=\pw(K_n)=\re(K_n)=h(n).$$
\qed
\end{proposition}

The parameter $h(n)$ was recently studied by Horvat, Pisanski
and \v Zitnik~\cite{Horvat.Pisanski.Zitnik.2009}.

In this paper we give some relationships between the three parameters.
After giving some basic properties and examples in Section~\ref{sec:basic}, we show in Section~\ref{sec:more} that the dilation coefficient and the plane-width are equivalent graph parameters in the sense that they are bounded on the same sets of graphs.
These two parameters are also equivalent to the chromatic
number, and are therefore bounded from above by a function of the maximum degree.
On the other hand, there exist cubic graphs with arbitrarily large resolution coefficient.
In particular, while the plane-width and the dilation coefficient are bounded from above by a function of the resolution coefficient, the converse is not true.

\begin{sloppypar}
In Section~\ref{sec:1dim} we examine the natural one-dimensional analogues of these three parameters (denoted by $\dc_1(G)$, $\pw_1(G)$ and $\re_1(G)$), and show that they coincide or almost coincide with three well studied graph parameters: the circular chromatic number $\chi_c(G)$, the chromatic number $\chi(G)$, and the bandwidth $\bw(G)$. As a corollary of some relations among
$\dc_1(G)$, $\pw_1(G)$ and $\re_1(G)$, we obtain independent proofs of some of the well-known relations between $\chi_c(G)$, $\chi(G)$ and $\bw(G)$.
\end{sloppypar}

In Section~\ref{sec:minor} we show that $G$ is planar whenever $\re(G)<\sqrt{2}$, while there exist graphs with $\re(G)= \sqrt{2}$ that contain arbitrarily large clique minors.

\section{Preliminaries}

In this section we review some definitions and basic results.
For terms left undefined, we refer the reader to~\cite{Diestel}.
As usual, let $\chi(G)$ denote the \emph{ chromatic number}, i.e.,
the least number of colors needed for a proper vertex coloring of $G$, and
$\omega(G)$ the \emph{ clique number}, i.e., the maximum size of a
complete subgraph of $G$. Clearly, $\omega(G) \le \chi(G)$ for any graph $G$.
Graphs $G$ such that for all induced subgraphs $H$ of $G$ the equality
$\omega(H) = \chi(H)$ holds are called \emph{perfect}.
Recently perfect graphs have been characterized by Chudnovsky, Robertson, Seymour
and Thomas~\cite{Chudnovsky.et.al.2006}, who proved
the famous Strong Perfect Graph Theorem conjectured by Berge in 1961~\cite{Berge}.

Some useful connections between the plane-width and the
chromatic number of a graph, for small values of these two
parameters, are given in~\cite{Kaminski.et.al.2009}:

\begin{thm}\label{thm:pw-chi}
For all graphs $G$,
\begin{enumerate}
\item[(a)] $\pw(G) = 1$ if and only if $\chi(G) \leq 3$\,,
\item[(b)] $\pw(G) \in (2 / \sqrt 3\,, \sqrt 2]$ if and only if $\chi(G) = 4$\,,
\item[(c)] $\pw(K_4) = \sqrt{2}$.
\qed
\end{enumerate}

\end{thm}

The following property of the plane-width will also be used in some of our proofs.

\begin{proposition}[\cite{Kaminski.et.al.2009}]      \label{prop:pw-chi}
Let $G$ be a graph such that $\chi(G)=\omega(G)$. Then,
$\pw(G)=h(\chi(G))$.
\qed
\end{proposition}

Obviously, the statement of the above proposition holds for perfect graphs.

\section{Basic properties and examples}\label{sec:basic}

The dilation coefficient, the plane-width and the resolution coefficient
of a graph  are related by the following inequalities.

\begin{proposition}\label{prop:1}
For every graph $G$,
$$1\le \dc(G)\le \min\{\pw(G),\re(G)\}\,.$$
\end{proposition}

\begin{sloppypar}\begin{proof}
The inequality $\dc(G)\ge 1$ is clear.

Let $\rho$ be a non-edge-degenerate representation of $G$. Then, since $E(G)\subseteq V_2(G)$, $$\frac{\max_{uv\in E(G)}d(\rho(u),\rho(v))}{\min_{uv\in E(G)}d(\rho(u),\rho(v))}\le \frac{\max_{uv\in V_2(G)}d(\rho(u),\rho(v))}{\min_{uv\in E(G)}d(\rho(u),\rho(v))}\,.$$
Taking the minimum over all such representations, it follows that $\dc(G)\le
\pw(G)$.

To see that $\dc(G)\le \re(G)$, let $\rho$ be a non-vertex-degenerate representation of $G$ that achieves the minimum in the definition of $\re(G)$. Then, since $\rho$ is
non-edge-degenerate and $E(G)\subseteq V_2(G)$, we have
$$\dc(G)\le \frac{\max_{uv\in E(G)}d(\rho(u),\rho(v))}{\min_{uv\in E(G)}d(\rho(u),\rho(v))}\le \frac{\max_{uv\in E(G)}d(\rho(u),\rho(v))}
{\min_{uv\in V_2(G)}d(\rho(u),\rho(v))}= \re(G)\,.$$
\end{proof}
\end{sloppypar}

\begin{sloppypar}
Clearly, the inequalities from Proposition~\ref{prop:1} are tight.
For instance, $\dc(G)=\pw(G)=\re(G)=1$ for every graph with
at most three vertices and at least one edge.
Moreover, the dilation coefficient of a graph $G$ is equal to~1
if and only if $G$ admits a representation with all edges of the same length.
Such graphs are known under the name of {\em unit distance graphs}.
There is a strong chemical motivation for exploring unit distance graphs
and related concepts, since relevant chemical graphs tend to have
all bond lengths of almost the same size.
By Theorem~\ref{thm:pw-chi}, graphs of plane-width 1 are precisely the 3-colorable
graphs.
To the best of our knowledge, graphs of unit resolution coefficient have not previously
appeared in the literature.
\end{sloppypar}

The following two examples, based on unit distance graphs, show that
$\dc(G)$ can be strictly smaller than $\min\{\pw(G),\re(G)\}$,
and that $\re(G)$ and $\pw(G)$ are incomparable.

\begin{example}
Let $G$ be the Moser spindle, see Fig.~\ref{fig:MS}. This is a 4-chromatic
unit distance graph on 7 vertices. Since $G$ is unit distance, we have
$\dc(G)=1$. On the other hand, since $G$ is 4-chromatic, $\pw(G)> 2/\sqrt{3}$ by Theorem~\ref{thm:pw-chi}. Furthermore, it is easy to see that $G$ cannot
be drawn in the plane with all edges of unit length and all the other pairs of
points at least unit distance apart. Therefore, $\re(G)>1$.
\end{example}

\begin{figure}[ht]
    \centering \includegraphics[width=30mm]{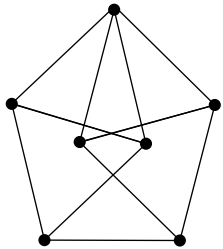}
\caption{A unit-distance representation of the Moser spindle}
\label{fig:MS}
\end{figure}

\begin{example}
Let $G$ be the graph depicted on Fig.~\ref{fig:G19}.
\begin{figure}[ht]
    \centering \includegraphics[width=40mm]{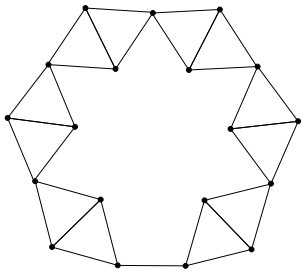}
\caption{A graph with $\re(G)<\pw(G)$ (all edge lengths are the same)}
\label{fig:G19}
\end{figure}

\begin{sloppypar}
Observe that $G$ is 4-chromatic, thus $\pw(G)> 2/\sqrt{3}$
by Theorem~\ref{thm:pw-chi}. The drawing from Fig.~\ref{fig:G19}
gives a representation $\rho$ of $G$ such that all edges are of unit length,
and all the other pairs of points are at least  unit distance apart.
Therefore,
$\max_{uv\in E(G)}d(\rho(u),\rho(v))=\min_{uv\in V_2(G)}d(\rho(u),\rho(v))$,
which implies that $\re(G)=1$.
This shows that in general, 
$\pw(G)$ is not bounded from above by $\re(G)$.
\end{sloppypar}

To see that also $\re(G)$ is not bounded from above by $\pw(G)$,
let $G$ be the 4-wheel, that is, the graph obtained from a 4-cycle by adding to it
a dominating vertex. Then, since $G$ is 3-colorable, $\pw(G)=1$ by
Theorem~\ref{thm:pw-chi}. However, it is easy to see that $G$ cannot be drawn
in the plane with all edges of the same length, which implies that $\re(G)> 1$.
\end{example}

In Section~\ref{sec:more} we will determine which of these three parameters is bounded from above by a function of another one.

A {\it homomorphism} of a graph $G$ to a graph $H$ is an adjacency-preserving mapping, that is a mapping $\phi:V(G)\to V(H)$ such that $\phi(u)\phi(v)\in E(H)$ whenever $uv\in E(G)$. We say that a graph $G$ is {\it homomorphic} to a graph $H$ if there exists a homomorphism of $G$ to $H$. A graph invariant $f$ is {\it homomorphism monotone} if $f(G)\le f(H)$ whenever $G$ is homomorphic to $H$. In~\cite{Kaminski.et.al.2009}, it was shown that the plane-width is homomorphism monotone.
We now show that the same property holds for the dilation coefficient, and that the resolution coefficient is monotone with respect to the subgraph relation.

\begin{proposition}\label{prop:subg} Let $G$ and $H$ be graphs with at least one edge.
\begin{itemize}
  \item[$(i)$] If $G$ is homomorphic to $H$ then $\dc(G)\le \dc(H)$.
  \item[$(ii)$] If $G$ is a subgraph of $H$ then $\re(G)\le \re(H)$.
\end{itemize}\end{proposition}

\begin{proof}
It follows from definitions that $\dc(G)$ is the minimum value of $p$
such that $G$ is homomorphic to some graph $H$ whose vertex set is a subset
of $\mathbb{R}^2$ and such  that every edge of $H$ connects two points
at Euclidean distance at least 1 and at most $p$. This observation together
with the transitivity of the homomorphism relation immediately implies $(i)$.

To see~$(ii)$, let $\rho$ be a non-vertex-degenerate representation of $H$ that
achieves the minimum in the definition of $\re(H)$. Then, since the restriction
of $\rho$ to $V(G)$ is a non-vertex-degenerate representation of $G$, $E(G)\subseteq E(H)$ and
$V_2(G)\subseteq V_2(H)$, we have
$$
\re(G)\le \frac{\max_{uv\in E(G)}d(\rho(u),\rho(v))}
               {\min_{uv\in V_2(G)}d(\rho(u),\rho(v))}\le
          \frac{\max_{uv\in E(H)}d(\rho(u),\rho(v))}
               {\min_{uv\in V_2(H)}d(\rho(u),\rho(v))}= \re(H)\,.
$$
\end{proof}

\section{More on relationship between $\dc$, $\pw$, and $\re$}\label{sec:more}
In this section, we examine more closely the relations between these three parameters.
First, we show in Theorem~\ref{thm:pwdelta} that the plane-width of a graph is
bounded from above by a function of its dilation coefficient. Our proof will make
use of the following result from~\cite{Kaminski.et.al.2009}
(combining Lemmas 2.2 and 3.7 therein).

\begin{lem}(\cite{Kaminski.et.al.2009})\label{lem:upperbound}
There exists a constant $C>0$ such that for every graph $G$,
$$\pw(G) \leq \sqrt{\frac{2\sqrt{3}}{\pi}\chi(G)} +C \,.$$
\end{lem}

\begin{thm}               \label{thm:pwdelta}
There exist positive constants $K, C>0$ such that for every graph $G$,
$$\pw(G)\le K\cdot\dc(G)+C\,.$$
\end{thm}
%

\begin{proof}
\begin{sloppypar}
Let $\rho$ be a representation of $G$ achieving the minimum in the definition
of the dilation coefficient (cf.~Equation~(\ref{eq:dc})).
We may assume, without loss of generality, that
$\min_{uv\in E(G)}d(\rho(u),\rho(v))=1$
(otherwise, we scale the representation accordingly). Hence
$\dc(G)$ equals the maximum length of an edge of $G$ w.r.t.~the representation $\rho$.
\end{sloppypar}

Let us cover the set $\{\rho(v):v\in V(G)\}$ with pairwise disjoint translates of the
half-open square $S=[0, t/\sqrt 2)\times[0, t/\sqrt 2)$, arranged in a grid-like
way, where $t=\lceil\sqrt 2\dc(G)\rceil+1$. Furthermore, we partition each copy
$S'$ of $S$ into $t^2$ pairwise disjoint translates $A_{ij}(S')$ of the set
$[0,1/\sqrt 2)\times[0,1/\sqrt 2)$, for each $i,j\in \{1,\ldots,t\}$. Here, the
pair $(i,j)$ denotes the ``coordinates" of the square $A_{ij}(S')$ within $S'$. We
do the assignment of the coordinate pairs to the small squares in the same way for
all copies $S'$ of $S$.

Let $c:V(G)\to \{1,\ldots,t\}^2$ be a coloring of the vertices of $G$ that assigns
to each $v\in V(G)$ the unique $(i,j)\in \{1,\ldots,t\}^2$ such that there exists
a translate $S'$ of $S$ such that $\rho(v)\in A_{ij}(S')$. By construction, $c$ is a
proper coloring of $G$. Therefore, $\chi(G)\le t^2$. By
Lemma~\ref{lem:upperbound}, $\pw(G) \leq \sqrt{\frac{2\sqrt{3}}{\pi}\chi(G)} +C
\le \sqrt{\frac{2\sqrt{3}}{\pi}}t +C$. Combining this inequality with the
inequality $t \le \sqrt 2\dc(G)+2$, the proposition follows:
we may take $K=\sqrt{\frac{4\sqrt 3}{\pi}}\approx 1.4850$.\footnote{A better constant $K=\sqrt{\frac{8\sqrt 3}{3\pi}}+\epsilon\le 1.2126$
can be obtained by covering the plane with hexagons instead of the squares
(similarly as was done in Lemma 3.6 in~\cite{Kaminski.et.al.2009}).}
\end{proof}

We say that two graph parameters $f$ and $g$ are {\it equivalent} if
they are bounded on precisely the same sets of graphs, that is, if
for every set of graphs ${\cal G}$, we have
$$
\sup\{f(G)~:~G\in {\cal G}\} <\infty
$$
if and only if
$$
\sup\{g(G)~:~G\in {\cal G}\} <\infty\,.
$$
It follows from Proposition~\ref{prop:1} and Theorem~\ref{thm:pwdelta}
that the dilation coefficient and the plane-width are equivalent graph parameters,
which, according to the following result from~\cite{Kaminski.et.al.2009},
are also equivalent to the chromatic number $\chi$:

\begin{thm}[\cite{Kaminski.et.al.2009}]                \label{thm:chi-pw}
For every $\epsilon>0$ there exists an integer $k$ such that for all
graphs $G$ of chromatic number at least $k$,
$$
\left(\frac{\sqrt 3}{2}-\epsilon\right)\sqrt{\chi(G)}
       < \pw(G)< \left(\sqrt{\frac{2\sqrt{3}}{\pi}}+\epsilon\right)\sqrt{\chi(G)}\,.
$$
\end{thm}

On the other hand, it turns out that the resolution coefficient is not equivalent to the dilation coefficient or to the plane-width.

\begin{thm}\label{prop:pwre}
For every positive function $f$ there exists a graph $G$ such that
$\re(G) > f(\pw(G))$.

There exists a function $f$ such that for every graph $G$,
$\pw(G)\le f(\re(G))$.
\end{thm}

\begin{proof}
Take $G=K_{1,n}$. Since these graphs are bipartite, $\pw(K_{1,n})=1$ for all $n$.
However, the values of $\re(K_{1,n})$ tend to infinity with increasing $n$.
This follows by observing that there exists a constant $C>0$ such that for every $N>0$, in any non-vertex-degenerate representation $\rho$ of $K_{1,n}$ with $\min_{uv\in V_2(G)}d(\rho(u),\rho(v))=1$,
less than $CN^2$ vertices can be mapped to distance at most $N$
from the image of the center of the star. Thus, $\re(K_{1,n}) = \Omega(\sqrt n)$. This shows the first part of the theorem.

Given a graph $G$ with at least one edge,
let $\Delta(G)$ denote its maximum vertex degree.
 Since $K_{1,\Delta(G)}$ is a subgraph of $G$, Proposition~\ref{prop:subg} implies that $\re(G)\ge \re(K_{1,\Delta(G)})= \Omega(\sqrt{\Delta(G)})$. Therefore, the maximum degree of a graph is bounded from above by a function of its resolution coefficient. In particular, since $\chi(G)\le \Delta(G)+1$,
the chromatic number of $G$ is also bounded from above by a function of $\re(G)$, and the second part of the theorem follows by Theorem~\ref{thm:pw-chi}.\end{proof}

In the above proof, large degree caused the resolution coefficient to be large.
The next example shows that large maximum degree is only a sufficient
but not a necessary condition for large resolution coefficient. Therefore, while the chromatic number of a graph is bounded from above by a function of its maximum vertex degree, and then the same is true for the dilation coefficient and the plane-width, this is not the case for the resolution coefficient.

\begin{thm}               \label{thm:degreere}
For each $R > 0$ there exists a graph $G_R$ of maximum degree at most 3
such that $\re(G_R) > R$.
\end{thm}
\begin{proof}
Let $H_k$ denote a full cubic tree with $k$ layers, rooted at vertex $v_0$.
Let $V_i$ denote the number of vertices of layer $i$.
Then $V_0=1$, $V_1=3$, $V_2=3\cdot 2$, \dots, $V_k=3\cdot 2^{k-1}$
and $|V(H_k)|=V_0+V_1+ \dots + V_k = 3 \cdot 2^k-2$.
Let $\rho$ be  a representation of $H_k$, where $\re(H_k)$ is achieved.
Let $r= \min_{uv\in V_2(H_k)}\{d(\rho(u),\rho(v))\}.$
There exists at least one vertex $x$ of $V_k$ such that
$d(\rho(v_0),\rho(x)) > C \cdot r\sqrt{3 \cdot 2^k-2}=:R^\prime$.
Since  the graph distance $d(V_0,x)=k$ there exists an edge $pq$
on the path from $v_0$ to $x$ such that
$$
d(\rho(p),\rho(q)) \ge R^\prime/k = \frac{C \cdot r\sqrt{3 \cdot 2^k-2}}{k}.
$$

Therefore
$$
\rho(H_k) \ge \frac{d(\rho(p),\rho(q))}{r}
          \ge \frac{C \cdot \sqrt{3 \cdot 2^k-2}}{k}.
$$

For a given $R$ define $k_0$ to be the smallest $k$ for which
$C\cdot (3 \cdot 2^{k}-2) \ge R$ and define
$G_R=H_{k_0}$. Clearly, $G_R$ has maximum degree at most 3 and $\re(G_R) > R$.
\end{proof}

\section{One-dimensional restrictions}\label{sec:1dim}

\noindent
It is interesting to consider the restrictions of these three parameters to a
line instead of the plane.  That is, every vertex gets mapped to a point on the
real line instead of being mapped to a point in $\mathbb{R}^2$ and the definitions
are analogous to those given by Equations (\ref{eq:dc})-(\ref{eq:res}).
We denote the corresponding parameters by $\dc_1(G)$, $\pw_1(G)$ and $\re_1(G)$.

It turns out that the one-dimensional variant of the dilation coefficient
is strongly related to the circular chromatic number of the graph,
the ``line-width" is  strongly related to the chromatic number of the graph,
while the one-dimensional variant of the resolution coefficient coincides with the graph's bandwidth, defined as
$$
\bw(G):=\min_{\pi:V\to\{1,\ldots,n\}, \textrm{bij.}}\max_{uv\in E}|\pi(u)-\pi(v)|\,,
$$
where $n=|V(G)|$.

Given a graph $G = (V,E)$ and $c\ge 1$, a {\it circular $c$-coloring} is a mapping
$f:V\to [0,c)$ such that for every edge $uv\in E$, it holds that $1\le |f(u)-f(v)|\le c-1$.
The {\it circular chromatic number} $\chi_c(G)$ of $G$ is defined as
$$\chi_c(G) = \inf\{c\,:\,G \textrm{ admits a circular $c$-coloring} \}\,.$$
It is known that the infimum in the definition above is always attained (see, e.g.,~\cite{Zhu.2001}).

\begin{thm}\label{thm-1dim}
For every graph $G$, the following holds:

      \begin{enumerate}
        \item[$(i)$] $\dc_1(G) = \chi_c(G)-1$
        \item[$(ii)$] $\pw_1(G) = \chi(G)-1$.
        \item[$(iii)$] $\re_1(G)=\bw(G)$.
      \end{enumerate}
\end{thm}

\begin{proof}
Let $G=(V,E)$ be a graph.

\begin{sloppypar} $(i)$
Let $c = \chi_c(G)$ and let $f$ be a circular $c$-coloring of $G$.
Then, $f$ defines a non-edge-degenerate one-dimensional representation of $G$ such that $\min_{uv\in E}|f(u)-f(v)| \ge 1$ and $\max_{uv\in E}|f(u)-f(v)| \le c-1$. It follows that
$\frac{\max_{uv\in E}|f(u)-f(v)|} {\min_{uv\in E}|f(u)-f(v)|}\le c-1$, implying $\dc_1(G) \le \chi_c(G)-1$.
\end{sloppypar}

\begin{sloppypar}
Conversely, let $\rho^*:V\to \mathbb{R}$ be a non-edge-degenerate
one-dimensional representation of $G$ that achieves the minimum in the definition of $\dc_1(G)$. Without loss of generality, we may assume that
$\min_{uv\in E}|\rho^*(u)-\rho^*(v)| = 1$ and that $\min_{v\in V}\rho^*(v) = 0$.
Let $c = \dc_1(G)+1$, and define $f:V\to \mathbb{R}$ as follows:
\end{sloppypar}

For all $v\in V$,
$$f(v) = \rho^*(v)-c\cdot\left\lfloor\frac{\rho^*(v)}{c}\right\rfloor\,.$$
To show that $\chi_c(G)\le \dc_1(G)+1 = c$, it suffices to verify that $f$ is a circular $c$-coloring of $G$. Since $0\le \frac{\rho^*(v)}{c}-\left\lfloor\frac{\rho^*(v)}{c}\right\rfloor<1$, $f$ maps vertices of $G$ to the interval $[0, c)$. It remains to show that for every edge $uv\in E$, it holds that $1\le |f(u)-f(v)|\le c-1$. Notice that by the choice of $\rho^*$, we have
$1\le |\rho^*(u)-\rho^*(v)|\le c-1$ for every edge $uv\in E$.

Let $uv\in E$. We may assume that $\rho^*(u)\le \rho^*(v)$. Let $k = \left\lfloor\frac{\rho^*(u)}{c}\right\rfloor$ and $\ell = \left\lfloor\frac{\rho^*(v)}{c}\right\rfloor$.
The inequality $\rho^*(u)\le \rho^*(v)$ implies that $k\le \ell$. Moreover, the inequality
$\rho^*(v)\le \rho^*(u)+c-1$ implies $\ell\le k+1$.

If $\ell=k$, then $f(u) -f(v) = \rho^*(u)-\rho^*(v)$ and the inequalities $1\le |f(u)-f(v)|\le c-1$ follow.

Suppose now that $\ell = k+1$. Then $f(u) -f(v) = \rho^*(u)-\rho^*(v)+c$. Therefore, it follows from $\rho^*(v)-\rho^*(u)\le c-1$ that $f(u) -f(v)\ge 1$, implying $f(u)-f(v)=|f(u) -f(v)|\ge 1$.
Similarly, it follows from $\rho^*(v)-\rho^*(u)\ge 1$ that $f(u) -f(v)=|f(u) -f(v)|\le c-1$.

\begin{sloppypar}
This shows that $f$ is  a circular $c$-coloring of $G$, which implies that $\chi_c(G)\le \dc_1(G)+1$.
\end{sloppypar}

\medskip
\begin{sloppypar}
$(ii)$
Any $k$-coloring of $G$ with colors in the set $\{1,\ldots, k\}\subseteq \mathbb{R}$
defines a non-edge-degenerate one-dimensional representation $\rho$ of $G$ such that
$\frac{\max_{uv\in V_2(G)}|\rho(u)-\rho(v)|} {\min_{uv\in E}|\rho(u)-\rho(v)|}\le k-1$.
Therefore, $\pw_1(G) \le \chi(G)-1$.
\end{sloppypar}

\begin{sloppypar}
Conversely, let $\rho^*:V\to \mathbb{R}$ be a non-edge-degenerate
one-dimensional representation of $G$
that achieves the minimum in the definition of $\pw_1(G)$.
Without loss of generality, we may assume that
$\min_{uv\in E}|\rho^*(u)-\rho^*(v)| = 1$ and that $\min_{v\in V}\rho^*(v) = 1$.
Define $f:V\to \mathbb{R}$ as follows: For all $v\in V$, let $f(v) = \lfloor \rho^*(v)\rfloor$.
Then, $f$ is a proper $k$-coloring of $G$, where $k = \lfloor\pw_1(G)\rfloor+1$.
Therefore, $\chi(G)\le \lfloor\pw_1(G)\rfloor+1\le \pw_1(G)+1$.
\end{sloppypar}

\medskip
\begin{sloppypar}
$(iii)$
Any bijective mapping $\pi:V\to\{1,\ldots,n\}\subseteq \mathbb{R}$ defines a
non-vertex-degenerate one-dimensional representation of $G$ such that
$$\frac{\max_{uv\in E}|\pi(u)-\pi(v)|} {\min_{uv\in V_2(G)}|\pi(u)-\pi(v)|} = \max_{uv\in E}|\pi(u)-\pi(v)|\,.$$ Therefore, $\re_1(G) \le \bw(G)$.
\end{sloppypar}

\begin{sloppypar}
Conversely, let $\rho^*:V\to \mathbb{R}$ be a non-vertex-degenerate
one-dimensional representation of $G$
that achieves the minimum in the definition of $\re_1(G)$.
Without loss of generality, we may assume that
$\min_{uv\in V_2(G)}|\rho^*(u)-\rho^*(v)| = 1$.
Then $$\re_1(G) = \frac{\max_{uv\in E}|\rho^*(u)-\rho^*(v)|} {\min_{uv\in V_2(G)}|\rho^*(u)-\rho^*(v)|} = \max_{uv\in E}|\rho^*(u)-\rho^*(v)|\,.$$ Therefore, to show that $\bw(G) \le \re_1(G)$, it suffices to show that
$\bw(G)\le \max_{uv\in E}|\rho^*(u)-\rho^*(v)|$.
\end{sloppypar}

Order the vertices of $V=\{v_1,\ldots, v_n\}$ according to the increasing values of their images: $$\rho^*(v_1)< \rho^*(v_2)<\cdots<\rho^*(v_n)\,,$$
and let $\pi(v_i) = i$ for all $i\in \{1,\ldots,n\}$.
This defines a bijective mapping $\pi:V\to \{1,\ldots,n\}$.

\begin{sloppypar}
Let $uv\in E$, and assume that $i := \pi(u)< j := \pi(v)$. By the definition of $\pi$ and since
$\min_{xy\in V_2(G)}|\rho^*(x)-\rho^*(y)| = 1$, we infer that  $|\rho^*(u)-\rho^*(v)|\ge j-i$.
On the other hand, $|\pi(u)-\pi(v)| = j-i$, which implies that
$|\pi(u)-\pi(v)|\le |\rho^*(u)-\rho^*(v)|$. Therefore, $\max_{uv\in E}|\pi(u)-\pi(v)|\le \max_{uv\in E}|\rho^*(u)-\rho^*(v)|$, implying
$\bw(G)\le \max_{uv\in E}|\pi(u)-\pi(v)|\le \max_{uv\in E}|\rho^*(u)-\rho^*(v)| = \re_1(G)$.
\end{sloppypar}
\end{proof}

\begin{proposition}\label{prop:1-dim}
For every graph $G$,
$$\lceil \dc_1(G)\rceil = \pw_1(G)\le \re_1(G)\,.$$
\end{proposition}

\begin{proof}
In the same way as the inequality  $\dc(G)\le \pw(G)$
from Proposition~\ref{prop:1}, one can prove the inequality
$\dc_1(G)\le \pw_1(G)$. Since $\pw_1(G) = \chi(G)-1$ by Theorem~\ref{thm-1dim},
$\pw_1(G)$ is always integral, therefore $\lceil \dc_1(G)\rceil \le \pw_1(G)$.

To see that $\pw_1(G)\le \lceil \dc_1(G)\rceil$, consider a non-edge-degenerate one-dimensional representation $\rho^*:V\to \mathbb{R}$ of $G$
that achieves the minimum in the definition of $\dc_1(G)$. Without loss of generality, we may assume that
$\min_{uv\in E}|\rho^*(u)-\rho^*(v)| = 1$ and that $\min_{v\in V}\rho^*(v) = 0$.
Let $k := \left\lceil \dc_1(G)\right\rceil+1$, and
define $\rho':V\to \{0,1,\ldots, k-1\}\subseteq \mathbb{R}$ as follows: For all $v\in V$, let $\rho'(v) = \left\lfloor \rho^*(v)\right \rfloor(\!\!\!\!\mod k)$.

 \begin{sloppypar}
For every $uv\in E$, we have $|\rho'(u)-\rho'(v)| \ge 1$. Indeed: if
$\rho'(u)=\rho'(v)$ then \hbox{$|\rho^*(u)-\rho^*(v)|\ge k> \dc_1(G)$}, contrary to the choice of $\rho^*$ and the definition of $k$. Therefore, $\rho'$ is a non-edge-degenerate one-dimensional representation of $G$, and since the image of $\rho'$ is contained in the set $\{0,1,\ldots, k-1\}$, we have  $\max_{uv\in V_2(G)}|\rho'(u)-\rho'(v)|\le k-1 = \left\lceil \dc_1(G)\right\rceil$, implying that $\pw_1(G)\le \left\lceil \dc_1(G)\right\rceil$.
\end{sloppypar}

In the same way as the inequality  $\dc(G)\le \re(G)$ from Proposition~\ref{prop:1}, one can prove the inequality $\dc_1(G)\le \re_1(G)$. Since $\re_1(G) = \bw(G)$ by Theorem~\ref{thm-1dim}, $\re_1(G)$ is always integral, and consequently $\lceil \dc_1(G)\rceil \le \re_1(G)$.
\end{proof}

Theorem~\ref{thm-1dim} and Proposition~\ref{prop:1-dim} provide an alternative proof of the following well known relations:
\begin{itemize}
  \item The equality $\chi(G) = \left\lceil\chi_c(G)\right\rceil$ showing that the circular chromatic number is the refinement of the chromatic number~\cite{Vince.1988}.
  \item The inequality $\bw(G) \ge\chi(G)-1$, proved in~\cite{Chvatalova.Dewdney.Gibbs.Korthage.1975}.
\end{itemize}

Theorems~\ref{thm:chi-pw} and~\ref{thm-1dim}  together with Proposition~\ref{prop:1-dim} also imply that both dilation coefficient and plane-width are equivalent to their one-dimensional counterparts. On the other hand, the resolution coefficient is
not equivalent to its one-dimensional variant, the bandwidth. There exist graphs of unit resolution and arbitrarily large bandwidth:

\begin{example}
Let $G$ be the Cartesian product of two paths, $G=P_m\Box P_n$. (For the definition of the Cartesian product, see e.g.~\cite{ImrichKlavzar}.) By a result of
Hare {\it et al.}~\cite{Hare.Hare.Hedetniemi.1985},  the bandwidth of $G$
is equal to $\min\{m,n\}$. Representing the vertices of $G$ as points
$\{1,\ldots,m\}\times\{1,\ldots,n\}$ of the integer grid $\mathbb{Z}^2$ shows
that the resolution coefficient of $G$ is equal to 1.

The {\it local density} of a graph is  defined as
$\max_{v,r}[|N(v,r)|/(2r)]$, where $N(v,r)$ denotes the set of all vertices
at distance at most $r$ from a vertex $v$.  While the local density is a
lower bound for the bandwidth, this example shows that the local density
(or any increasing function of it) is not a lower bound for the resolution coefficient.
\end{example}

\section{Minor-closed classes}\label{sec:minor}

A graph $G$ is a {\it minor} of a graph $H$ if $G$ can be obtained from $H$ by a sequence of vertex deletions, edge deletions, and edge contractions. We say that a class of graphs is minor closed if it contains together with any graph $G$ all minors of $G$. In this section, we prove some results that link the resolution coefficient to minor-closed graph classes. First, we observe that all graphs of small enough resolution coefficient are planar.

\begin{proposition}\label{prop:re-planar}
If $\re(G) < \sqrt{2}$ then $G$ is planar.
\end{proposition}

\begin{proof}
Let $G = (V,E)$ such that $\re(G) < \sqrt{2}$, and consider a representation $\rho^*:V\to\mathbb{R}^2$ achieving the minimum in the definition of the resolution coefficient. As usual, we assume that
$\min_{uv\in V_2(G)}d(\rho^*(u),\rho^*(v))=1$.

First, observe that since $\re(G)< 2$, for every vertex $v\in V$ and every edge $e\in E$, it holds that $\rho^*(v)$ is not contained in the interior of the line segment $\rho^*(e)$.
Therefore, since $\rho^*$ is non-vertex-degenerate and maps no vertex to an interior of a line segment connecting the endpoints of an edge, it defines a drawing of $G$ in the plane with straight line segments. We claim that this is a planar drawing. Suppose not, and let $uv$ and $xy$ be two distinct edges of $G$ such that the interiors of line segments $\rho^*(u)\rho^*(v)$ and $\rho^*(x)\rho^*(y)$ have a point in common. Basic geometric arguments show that a longest diagonal in a convex 4-gon in the plane is at least a factor of $\sqrt{2}$ longer than its shortest edge. Therefore, we conclude that at least one of the two line segments $\rho^*(u)\rho^*(v)$ and $\rho^*(x)\rho^*(y)$ has length $\sqrt{2}$ or more, contrary to the assumption that $\re(G)<\sqrt{2}$.
\end{proof}

\begin{cor}\label{cor:re=dc=1<pw}
Every graph $G$ such that $\re(G) = \dc(G) =1< \pw(G)$ is a 4-chromatic planar unit distance graph.
\end{cor}

\begin{proof}
$G$ is unit distance since $\dc(G) = 1$. By Proposition~\ref{prop:re-planar}, $G$ is planar.
Since $\pw(G)>1$, Theorem~\ref{thm:pw-chi} implies that $G$ is not 3-colorable.
Finally, since $G$ is planar, it must be 4-chromatic by the Four Color Theorem.
\end{proof}

Notice that the inclusion is proper: there exist 4-chromatic planar unit distance graphs that are not of unit resolution coefficient, for example the Moser spindle (cf.~Example 1).

It turns out that the statement of Proposition~\ref{prop:re-planar} cannot be generalized to surfaces of higher genus, in the sense that the genus of a graph would be bounded from above by a function of its resolution coefficient. In fact, we show that:

\begin{thm}
For every $n$ there exists a graph $G$ with $\re(G) = \sqrt 2$ and such that $K_n$ is a minor of $G$.\end{thm}

\begin{proof}
Let $n\ge 5$. We will construct a graph $G$ of maximum degree 3 with a $K_n$-minor and such that $\re(G)= \sqrt 2$.

First, we replace every vertex $v$ of $K_n$ with a cycle on $n-1$ vertices and connect each vertex of the cycle to a different neighbor of $v$ in $K_n$. Repeating this procedure for all vertices of $G$, we obtain a cubic graph $G'$ which can be contracted to $K_n$.

We now place the vertices of $G'$ into the plane so that all point coordinates are integer and no two vertices are mapped to the same point. Moreover, we draw the edges between them so that: (i) each edge is represented by a rectilinear curve consisting of finitely many vertical and horizontal line segments, (ii) no two edges share a common line segment, (iii) at every point where two edges cross, they cross properly (i.e.,  they do not touch each other), and (iv) no vertex is contained in the interior of an edge. This can be done since the graph is of maximum degree~3. Moreover, all the coordinates of the breakpoints (points where a rectilinear curve bends) can be chosen to be rational-valued. By scaling the drawing appropriately, we then obtain a drawing of $G'$ such that all vertex coordinates as well as all coordinates of the crossing points and breakpoints are integer multiples of~4.

We now modify the obtained drawing by making a local modification at each crossing point. We essentially rotate each cross by $45^\circ$, as shown in Figure~\ref{fig:crossing}.

\begin{figure}[ht]
    \centering \includegraphics[width=50mm]{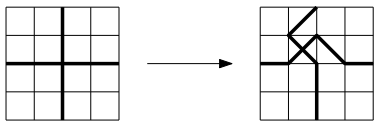}
\caption{The local modification at a crossing point; the grid represents a portion of the two-dimensional integer grid around the crossing point.}
\label{fig:crossing}
\end{figure}

Due to the assumption that in the previous drawing, horizontal and vertical lines connect points in $(4\mathbb{Z})^2$, performing such a local modification at each crossing point will not introduce any further crossings or touchings.

The resolution coefficient of $G'$ may be large. We now complete the proof by modifying $G'$ into a graph $G$ by subdividing edges. For every edge $e$ of $G'$, let $\ell(e)$ denote the curve representing $e$ in the above (modified) drawing. We insert a new vertex on each point in $\ell(e)\cap \mathbb{Z}^2$ that is not yet occupied by a vertex. We repeat this procedure for every edge $e$ of $G'$, and call the resulting graph $G$. Notice that the resulting drawing of $G$ defines a non-vertex-degenerate representation $\rho$ of $G$.

\begin{sloppypar}
Since all vertex coordinates are integer, we have $\min_{uv\in V_2(G)}d(\rho(u),\rho(v)) \ge 1$. On the other hand, due to the newly introduced vertices, it also holds that $\max_{uv\in E(G)}\{d(\rho(u),\rho(v))\}\le \sqrt{2}$. Therefore, the resolution coefficient of $G$ is at most $\sqrt{2}$. On the other hand, since $G$ contains $K_5$ as a minor, it is not planar, and hence $\re(G)\ge \sqrt{2}$ by Proposition~\ref{prop:re-planar}.
\end{sloppypar}

Finally, notice that as no new vertices were introduced at crossing points, $G$ can be contracted to $G'$ and hence to $K_n$, which implies that $G$ contains $K_n$ as a minor. \end{proof}

\section{Concluding remarks}
As we have shown, the relationship between the dilation coefficient, the plane-width and the resolution coefficient is non-trivial in many respects.
Although all these parameters somehow reflect the departure from unit-distance representations they do not behave in a uniform way.
The dilation coefficient and the plane-width are both equivalent to each other and also to their one-dimensional counterparts (in the sense that these parameters are bounded on precisely the same sets of graphs), while the resolution coefficient is not equivalent to either plane-width or dilation coefficient and also not to its one-dimensional analogue, the bandwidth.

There are several possibilities for further investigations. For instance, one could study the parameters $\dc$, $\pw$ and $\re$ in a more general framework where instead of $\mathbb{R}^2$ the host space is an arbitrary metric space; we have explored in this paper the case when the metric space is the real line. A related notion of colorings in distance spaces was studied by Mohar~\cite{Mohar.2001}.

There are several families of graphs for which at least two out of the three above parameters coincide. For example, the dilation coefficient and the plane-width coincide for all graphs $G$ such that $\chi(G)\in\{2,3,\omega(G)\}$. All three parameters coincide for complete graphs, for subgraphs of the Cartesian product of two paths, and for subgraphs of the strong product of two paths that contain a $K_4$. It would be interesting to get a complete classification of families of graphs for which at least two out of the three
parameters coincide.

Last but not least, we feel it would also be interesting to consider other measures of degeneracy of graph representations, such as the smallest angle or the smallest distance between an edge and vertices non-incident to it.

\section*{Acknowledgements}
This work was supported in part by
the research program P1-0294 of the Slovenian Agency for Research.

\bibliographystyle{plain}

\begin{thebibliography}{10}
\bibitem{Bateman.Erdos.1951}
P.~Bateman and P.~Erd\H{o}s.
\newblock Geometrical extrema suggested by a lemma of Besicovitch.
\newblock {\em American Math. Monthly}, 58:306--314, 1951.

\bibitem{Berge}
C.~Berge.
 \newblock F\"arbung von Graphen, deren s\"amtliche bzw.~deren ungerade Kreise starr sind,
 \newblock {\em Wiss. Z. Martin-Luther-Univ. Halle-Wittenberg Math.-Natur. Reihe} 10:114, 1961.

\bibitem{Bezdek.Fodor.1999}
A.~Bezdek and F.~Fodor.
\newblock Minimal diameter of certain sets in the plane.
\newblock {\em J.~Comb. Theory, Ser.~A}, 85:105--111, 1999.

\bibitem{Chudnovsky.et.al.2006}
M.~Chudnovsky, N.~Robertson, P.~Seymour and R.~Thomas,
\newblock The strong perfect graph theorem.
\newblock {\em Ann. of Math.} 164:51--229, 2006.

\bibitem{Chvatalova.Dewdney.Gibbs.Korthage.1975}
J.~Chvatalova, A.K.~Dewdney, N.E.~Gibbs and R.R.~Korfhage.
\newblock The bandwidth problem for graphs: a collection of recent results.
\newblock Research Report \#24, Department of Computer Science, UWO, London, Ontario (1975).

\bibitem{Diestel} R.~Diestel,
\newblock {\em Graph Theory. Third Edition.}
\newblock Springer-Verlag, Heidelberg, 2005.

\bibitem{Hare.Hare.Hedetniemi.1985}
W.R.~Hare, E.O'M.~Hare and S.T.~Hedetniemi,
\newblock Bandwidth of grid graphs.
\newblock Proceedings of the Sundance conference on combinatorics and related topics (Sundance, Utah, 1985). {\em Congr. Numer.} 50:67--76, 1985.

\bibitem{Horvat.Pisanski.Zitnik.2009}
B.~Horvat, T.~Pisanski and A.~\v Zitnik.
\newblock The dilation coefficient of a complete graph.
\newblock {\em Croat. Chem. Acta}, 82:771--779, 2009.

\bibitem{ImrichKlavzar}
W.~Imrich, S.~Klav\v{z}ar.
\newblock {\em Product graphs.}
\newblock Wiley-Interscience, New York, 2000.

\bibitem{Kaminski.et.al.2009}
M.~Kami\'nski, P.~Medvedev and M.~Milani\v c.
\newblock The plane-width of graphs.
\newblock {\em Submitted.} 2009.

\bibitem{Mohar.2001}
B. Mohar.
\newblock Chromatic number of a nonnegative matrix, in preparation
\newblock IMFM Preprint, 39:785, 2001.
{\tt http://www.imfm.si/preprinti/PDF/00785.pdf}

\bibitem{Pisanski.Zitnik.2008}
T.~Pisanski and A.~\v Zitnik.
\newblock Representing graphs and maps, Chapter in {\em Topics in Topological Graph Theory}, Series: Encyclopedia of Mathematics and its Applications (No. 129).
\newblock Cambridge University Press, 2009.

\bibitem{Schurmann.2002}
A.~Sch\"urmann.
\newblock On extremal finite packings.
\newblock {\em Discrete Comput.~Geom.}, 28:389--403, 2002.

\bibitem{Vince.1988}
A.~Vince.
\newblock Star chromatic number.
\newblock {\em J. Graph Theory}, 12:551--559, 1988.

\bibitem{Zhu.2001}
X.~Zhu.
\newblock Circular chromatic number: a survey.
\newblock {\em Discrete Math.}, 229:371--410, 2001.
\end{thebibliography}

\end{document}